\documentclass[a4paper,12pt,reqno]{amsart}
\usepackage[english]{babel}
\usepackage[T1]{fontenc}
\usepackage[left=1in,right=1in,top=1.2in,bottom=1.2in]{geometry}
\usepackage{times}
\usepackage{microtype}

\usepackage{commath,amssymb,amscd,mathrsfs,mathtools,dsfont,bbm}
\usepackage[all]{xy}
\usepackage{enumerate}
\usepackage{longtable}

\usepackage{cite}			

\usepackage[usenames,dvipsnames,svgnames,table]{xcolor}
\definecolor{red}{RGB}{255,25,25}
\definecolor{blue}{RGB}{25,50,200}
\usepackage{tikz-cd}
\usepackage[pagebackref,linktocpage]{hyperref}

\hypersetup{
pdftitle={},				
pdfauthor={Fei Hu},		
colorlinks=true,			
linkcolor=red,			
citecolor=MidnightBlue,	
filecolor=magenta,		
urlcolor=MidnightBlue	
}

\usepackage{cleveref} 

\newtheorem{theorem}{Theorem}[section]
\crefname{theorem}{Theorem}{Theorems}
\newtheorem{lemma}[theorem]{Lemma}
\crefname{lemma}{Lemma}{Lemmas}

\crefname{proposition}{Proposition}{Propositions}

\crefname{prop}{Proposition}{Propositions}

\crefname{corollary}{Corollary}{Corollaries}

\crefname{cor}{Corollary}{Corollaries}

\crefname{conjecture}{Conjecture}{Conjectures}

\crefname{conj}{Conjecture}{Conjectures}
\newtheorem*{conj*}{Conjecture}
\crefname{conj}{Conjecture}{Conjectures}

\theoremstyle{definition}
\newtheorem{definition}[theorem]{Definition}
\crefname{definition}{Definition}{Definitions}

\crefname{defn}{Definition}{Definitions}

\crefname{example}{Example}{Examples}

\crefname{notation}{Notation}{Notation}
\newtheorem*{notation*}{Notation}
\crefname{notation}{Notation}{Notation}

\crefname{problem}{Problem}{Problems}
\newtheorem{question}[theorem]{Question}
\crefname{question}{Question}{Questions}

\crefname{condition}{Condition}{Conditions}

\crefname{assumption}{Assumption}{Assumptions}

\theoremstyle{remark}

\crefname{rmk}{Remark}{Remarks}
\newtheorem*{rmk*}{Remark}
\crefname{rmk}{Remark}{Remarks}
\newtheorem{remark}[theorem]{Remark}
\crefname{remark}{Remark}{Remarks}

\crefname{fact}{Fact}{Facts}
\newtheorem{claim}[theorem]{Claim}
\crefname{claim}{Claim}{Claims}
\newtheorem*{claim*}{Claim}
\crefname{claim}{Claim}{Claims}

\crefname{step}{Step}{Steps}
\newtheorem{case}{Case}
\crefname{case}{Case}{Cases}

\numberwithin{equation}{section}

\newcommand{\what}[1]{\widehat{#1}}

\newcommand{\ol}[1]{\overline{#1}}

\newcommand{\longinjmap}{\lhook\joinrel\longrightarrow}

\newcommand{\lra}{\longrightarrow}

\newcommand{\arxiv}[1]{\href{https://arxiv.org/abs/#1}{{\tt arXiv:#1}}}
\newcommand{\doi}[1]{\href{https://doi.org/#1}{{\tt doi:#1}}}
\def\MR#1{\href{http://www.ams.org/mathscinet-getitem?mr=#1}{MR#1}}

\newcommand{\bA}{\mathbf{A}}
\newcommand{\bB}{\mathbf{B}}
\newcommand{\bC}{\mathbf{C}}

\newcommand{\bF}{\mathbf{F}}

\newcommand{\bH}{\mathbf{H}}
\newcommand{\bI}{\mathbf{I}}

\newcommand{\bM}{\mathbf{M}}

\newcommand{\bQ}{\mathbf{Q}}
\newcommand{\bR}{\mathbf{R}}
\newcommand{\bZ}{\mathbf{Z}}

\newcommand{\be}{\mathbf{e}}
\newcommand{\bk}{\mathbf{k}}
\newcommand{\bu}{\boldsymbol{u}}
\newcommand{\bv}{\boldsymbol{v}}

\newcommand{\sH}{\mathscr{H}}

\newcommand{\sL}{\mathscr{L}}

\newcommand{\sT}{\mathsf{T}}

\newcommand{\alg}{{\rm alg}}
\newcommand{\reduced}{{\rm red}}

\newcommand{\Alb}{\operatorname{Alb}}

\newcommand{\Aut}{\operatorname{Aut}}

\newcommand{\ch}{\operatorname{char}}
\newcommand{\CH}{\operatorname{CH}}

\newcommand{\End}{\operatorname{End}}
\newcommand{\et}{{\textrm{\'et}}}

\newcommand{\Hom}{\operatorname{Hom}}
\newcommand{\id}{\operatorname{id}}
\newcommand{\im}{\operatorname{Im}}
\newcommand{\isom}{\simeq}

\newcommand{\Mat}{\operatorname{M}}

\newcommand{\Nm}{\operatorname{N}}
\newcommand{\NS}{\operatorname{NS}}
\newcommand{\num}{\equiv}

\newcommand{\Pic}{\operatorname{Pic}}

\newcommand{\re}{\operatorname{Re}}

\newcommand{\vect}{\operatorname{vec}}

\begin{document}

\title[Cohomological and numerical dynamical degrees]{Cohomological and numerical dynamical degrees on abelian varieties}

\author{Fei Hu}
\address{Department of Mathematics, University of British Columbia, 1984 Mathematics Road, Vancouver, BC V6T 1Z2, Canada
\endgraf Pacific Institute for the Mathematical Sciences, 2207 Main Mall, Vancouver, BC V6T 1Z4, Canada}
\email{\href{mailto:hf@u.nus.edu}{\tt hf@u.nus.edu}}
\urladdr{\url{https://sites.google.com/view/feihu90s/}}

\begin{abstract}
We show that for a self-morphism of an abelian variety defined over an algebraically closed field of arbitrary characteristic, the second cohomological dynamical degree coincides with the first numerical dynamical degree.
\end{abstract}

\subjclass[2010]{
14G17,	
14K05,	
16K20.	
}


\keywords{dynamical degree, abelian variety, endomorphism algebra, \'etale cohomology, algebraic cycle, positive characteristic}

\thanks{The author was partially supported by a UBC-PIMS Postdoctoral Fellowship.}

\maketitle



\section{Introduction}
\label{section-intro}


\noindent
Let $X$ be a smooth projective variety defined over an algebraically closed field $\bk$, and $f$ a surjective morphism of $X$ to itself.
Inspired by Esnault--Srinivas \cite{ES13} and Truong \cite{Truong1611}, we associate to this map two dynamical degrees as follows.
Let $\ell$ be a prime different from the characteristic of $\bk$.
As a consequence of Deligne \cite{Deligne74} and Katz--Messing \cite{KM74}, the characteristic polynomial of $f$ on the $\ell$-adic \'etale cohomology group $H^i_{\et}(X, \bQ_\ell)$ is independent of $\ell$, and has integer coefficients, and algebraic integer roots
(cf.~\cite[Proposition~2.3]{ES13}; see also \cite{Kleiman68}).
The {\it $i$-{th} cohomological dynamical degree} $\chi_i(f)$ of $f$ is then defined as the spectral radius of the pullback action $f^*$ on $H^i_{\et}(X, \bQ_\ell)$, i.e.,
\[
\chi_i(f) = \rho\big(f^* \big|_{H^i_{\et}(X, \bQ_\ell)}\big).
\]
Alternatively, one can also define dynamical degrees using algebraic cycles.
Indeed, let $N^k(X)$ denote the group of algebraic cycles of codimension $k$ modulo numerical equivalence.
Note that $N^k(X)$ is a finitely generated free abelian group (cf.~\cite[Theorem~3.5]{Kleiman68}), and hence the characteristic polynomial of $f$ on $N^k(X)$ has integer coefficients and algebraic integer roots.
We define the {\it $k$-{th} numerical dynamical degree} $\lambda_k(f)$ of $f$ as the spectral radius of the pullback action $f^*$ on $N^k(X)_\bR \coloneqq N^k(X) \otimes_\bZ \bR$, i.e.,
\[
\lambda_k(f) = \rho\big(f^* \big|_{N^k(X)_\bR}\big).
\]

When $\bk \subseteq \bC$, we may associate to $(X, f)$ a projective (and hence compact K\"ahler) manifold $X_\bC$ and a surjective holomorphic map $f_\bC$. Then by the comparison theorem and Hodge theory, it is not hard to show that $\chi_{2k}(f) = \lambda_k(f)$; both of them also agree with the usual dynamical degree defined by the Dolbeault cohomology group $H^{k,k}(X_\bC, \bC)$ in the context of complex dynamics (see e.g. \cite[\S4]{DS17}).

For an arbitrary algebraically closed field $\bk$ (in particular, of positive characteristic), Esnault and Srinivas \cite{ES13} proved that for an automorphism of a smooth projective surface, the second cohomological dynamical degree coincides with the first numerical dynamical degree.
Their proof relies on the Enriques--Bombieri--Mumford classification of surfaces in arbitrary characteristic.
In general, Truong \cite{Truong1611} raised the following question (among many others).


\begin{question}[{cf.~\cite[Question~2]{Truong1611}}]
\label{qn:est}
Let $X$ be a smooth projective variety defined over an algebraically closed field $\bk$, and $f$ a surjective morphism of $X$ to itself.
Then is $\chi_{2k}(f) = \lambda_k(f)$ for any $1\le k\le \dim X$?
\end{question}

The above question turns out to be related to Weil's Riemann hypothesis (proved by Deligne in the early 1970s).
More precisely, when $X_0$ is a smooth projective variety defined over a finite field $\bF_q$, we let $X$ denote the base change of $X_0$ to the algebraic closure $\ol \bF_q$ of $\bF_q$
and let $F$ denote the Frobenius endomorphism of $X$ (with respect to $\bF_q$).
Then Deligne's celebrated theorem asserts that all eigenvalues of $F^*|_{H^i_{\et}(X, \bQ_\ell)}$ are algebraic integers of modulus $q^{i/2}$ (cf.~\cite[Th\'eor\`eme~1.6]{Deligne74}).
In particular, we have $\chi_i(F) = q^{i/2}$.
On the other hand, the $k$-{th} numerical dynamical degree $\lambda_k(F)$ of $F$ is equal to $q^k$.
See \cite[\S4]{Truong1611} for more details.

Truong proved in \cite{Truong1611} a slightly weaker statement that
\[
h_{\et}(f) \coloneqq \max_{i} \log \chi_i(f) = \max_{k} \log \lambda_k(f) \eqqcolon h_{\alg}(f),
\]
which is enough to conclude that the (\'etale) entropy $h_{\et}(f)$ coincides with the algebraic entropy $h_{\alg}(f)$ in the sense of \cite[\S6.3]{ES13}.
As a consequence, the spectral radius of the action $f^*$ on the even degree \'etale cohomology $H^{2\bullet}_{\et}(X, \bQ_\ell)$ is the same as the spectral radius of $f^*$ on the total cohomology $H^{\bullet}_{\et}(X, \bQ_\ell)$.\footnote{Recently, this was reproved by Shuddhodan \cite{Shuddhodan19} using a number-theoretic method, where the author introduced a zeta function $Z(X,f,t)$ for a dynamical system $(X, f)$ defined over a finite field.}
Note that when $\bk \subseteq \bC$, by the fundamental work of Gromov \cite{Gromov03} and Yomdin \cite{Yomdin87}, the algebraic entropy is also equal to the topological entropy $h_{\textrm{top}}(f_\bC)$ of the topological dynamical system $(X_\bC, f_\bC)$; see \cite[\S4]{DS17} for more details.

In this article, we give an affirmative answer to \cref{qn:est} in the case that $X$ is an abelian variety and $k=1$.

\begin{theorem}
\label{thmA}
Let $X$ be an abelian variety defined over an algebraically closed field $\bk$, and $f$ a surjective self-morphism of $X$.
Then $\chi_2(f) = \lambda_1(f)$.
\end{theorem}

\begin{remark}
\label{rmkA}
\begin{enumerate}[(1)]
\item When $f$ is an automorphism of an abelian surface $X$, the theorem was already known by Esnault and Srinivas (cf.~\cite[\S4]{ES13}).
Even in this two dimensional case, their proof is quite involved.
Actually, after a standard specialization argument, they applied the celebrated Tate theorem \cite{Tate66} (see also \cite[Appendix~I, Theorem~3]{Mumford}),
which asserts that the minimal polynomial of the geometric Frobenius endomorphism is a product of distinct monic irreducible polynomials.
Then they had four cases to analyze according to its irreducibility and degree.
Our proof is more explicit in the sense that we will eventually determine all eigenvalues of $f^*|_{N^1(X)_\bR}$.
\item Because of the lack of an explicit characterization of higher-codimensional cycles (up to numerical equivalence) like the N\'eron--Severi group $\NS(X)$ sitting inside the endomorphism algebra $\End^0(X)$, it would be very interesting to consider the case $k\ge 2$ next.
\end{enumerate}
\end{remark}


\section{Preliminaries on abelian varieties}
\label{section-prelim}


\noindent
We refer to \cite{Mumford} and \cite{Milne86} for standard notation and terminologies on abelian varieties.

\begin{notation*}
\label{notation}
The following notation remains in force throughout the rest of this article unless otherwise stated.
\renewcommand*{\arraystretch}{1.1}
\begin{longtable}{p{2cm} p{12cm}}
$\bk$ & an algebraically closed field of arbitrary characteristic \\
$\ell$ & a prime different from $\ch\bk$ \\
$X$ & an abelian variety of dimension $g$ defined over $\bk$ \\
$\what X$ & the dual abelian variety $\Pic^0(X)$ of $X$ \\
$\alpha, \  \psi$ & endomorphisms of $X$ \\
$\what \alpha, \  \what \psi$ & the induced dual endomorphisms of $\what X$ \\
$\End(X)$ & the endomorphism ring of $X$ \\
$\End^0(X)$ & $\End(X) \otimes_\bZ \bQ$, the endomorphism $\bQ$-algebra of $X$ \\
$\End(X)_\bR$ & $\End(X) \otimes_\bZ \bR = \End^0(X) \otimes_\bQ \bR$, the endomorphism $\bR$-algebra of $X$ \\
$\Mat_n(R)$ & the ring of all $n\times n$ matrices with entries in a ring $R$ \\
$\phi_\sL$ & the induced homomorphism of a line bundle $\sL$ on $X$: \\
 & \quad \quad $\phi_\sL \colon X \lra \what X, \ \ x \longmapsto t_x^*\sL \otimes \sL^{-1}$ \\
$\phi = \phi_{\sL_0}$ & a fixed polarization of $X$ induced from some ample line bundle $\sL_0$ \\
$^\dagger$ & the Rosati involution on $\End^0(X)$ defined in the following way: \\
 & \quad \quad $\psi \longmapsto \psi^\dagger \coloneqq \phi^{-1}\circ \what \psi \circ \phi$, for any $\psi \in \End^0(X)$ \\
$\NS(X)$ & $\Pic(X)/\Pic^0(X)$, the N\'eron--Severi group of $X$ \\
$\NS^0(X)$ & $\NS(X) \otimes_\bZ \bQ = N^1(X)_\bQ = \NS(X)_\bQ$ (see \cref{rmk:div-equiv-relation}) \\
$\NS(X)_\bR$ & $\NS(X) \otimes_\bZ \bR = \NS^0(X) \otimes_\bQ \bR = N^1(X)_\bR$ \\ 
$N^k(X)_\bR$ & $N^k(X) \otimes_\bZ \bR$, the $\bR$-vector space of numerical equivalent classes of \\
 & codimension-$k$ cycles (with $0\le k\le g = \dim X$) \\
$H^i_{\et}(X, \bQ_\ell)$ & $H^i_{\et}(X, \bZ_\ell) \otimes_{\bZ_\ell} \bQ_\ell$, the $\ell$-adic \'etale cohomology group of degree $i$ \\
$T_\ell X$ & the Tate module $\varprojlim_n X_{\ell^n}(\bk)$ of $X$, a free $\bZ_\ell$-module of rank $2g$ \\
$T_\ell \alpha$ & the induced endomorphism on $T_\ell X$ \\
$A$ & a simple abelian variety defined over $\bk$ \\
$D$ & $\End^0(A)$, the endomorphism $\bQ$-algebra of $A$ \\
$K$ & the center of the division ring $D = \End^0(A)$ \\
$K_0$ & the maximal totally real subfield of $K$ \\
$\bH$ & the standard quaternion algebra over $\bR$
\end{longtable}
\end{notation*}

For the convenience of the reader, we include several important structure theorems on the \'etale cohomology groups, the endomorphism algebras and the N\'eron--Severi groups of abelian varieties.
We refer to \cite[\S19-21]{Mumford} for more details.

First, the \'etale cohomology groups of abelian varieties are simple to describe.

\begin{theorem}[{cf.~\cite[Theorem~15.1]{Milne86}}]
\label{thm:etale-coh}
Let $X$ be an abelian variety of dimension $g$ defined over $\bk$, and let $\ell$ be a prime different from $\ch\bk$. Let $T_\ell X \coloneqq \varprojlim_n X_{\ell^n}(\bk)$ be the Tate module of $X$, which is a free $\bZ_\ell$-module of rank $2g$.
\begin{itemize}
\item[(a)] There is a canonical isomorphism
\[
H^1_{\emph\et}(X, \bZ_\ell) \isom \Hom_{\bZ_\ell}(T_\ell X, \bZ_\ell).
\]
\item[(b)] The cup-product pairing induces isomorphisms
\[
\bigwedge\nolimits^i H^1_{\emph\et}(X, \bZ_\ell) \isom H^i_{\emph\et}(X, \bZ_\ell),
\]
for all $i$. In particular, $H^i_{\emph\et}(X, \bZ_\ell)$ is a free $\bZ_\ell$-module of rank $\displaystyle \binom{2g}{i}$.
\end{itemize}
\end{theorem}

Furthermore, the functor $T_\ell$ induces an $\ell$-adic representation of the endomorphism algebra. In general, we have:

\begin{theorem}[{cf.~\cite[\S19, Theorem~3]{Mumford}}]
\label{thm:l-adic-rep}
For any two abelian varieties $X$ and $Y$, the group $\Hom(X,Y)$ of homomorphisms of $X$ into $Y$ is a finitely generated free abelian group, and the natural homomorphism of $\bZ_\ell$-modules
\[
\Hom(X, Y) \otimes_\bZ \bZ_\ell \lra \Hom_{\bZ_\ell}(T_\ell X, T_\ell Y)
\]
induced by $T_\ell \colon \Hom(X, Y) \lra \Hom_{\bZ_\ell}(T_\ell X, T_\ell Y)$ is injective.
\end{theorem}

For a homomorphism $f \colon X \lra Y$ of abelian varieties, its {\it degree} $\deg f$ is defined to be the order of the kernel $\ker f$, if it is finite, and $0$ otherwise.
In particular, the degree of an isogeny is always a positive integer.

\begin{theorem}[{cf.~\cite[\S19, Theorem~4]{Mumford}}]
\label{thm:char-poly}
For any $\alpha \in \End(X)$, there is a unique monic polynomial $P_\alpha(t) \in \bZ[t]$ of degree $2g$ such that $P_\alpha(n) = \deg(n_X - \alpha)$ for all integers $n$.
Moreover, $P_\alpha(t)$ is the characteristic polynomial of $\alpha$ acting on $T_\ell X$, i.e., $P_\alpha(t) = \det(t - T_\ell \alpha )$, and $P_\alpha(\alpha) = 0$ as an endomorphism of $X$.
\end{theorem}

We call $P_\alpha(t)$ as in \cref{thm:char-poly} the {\it characteristic polynomial of $\alpha$}.
On the other hand, we can assign to each $\alpha$ the characteristic polynomial $\chi_\alpha(t)$ of $\alpha$ as an element of the semisimple $\bQ$-algebra $\End^0(X)$.
Namely, we define $\chi_\alpha(t)$ to be the characteristic polynomial of the left multiplication $\alpha_L \colon \beta \mapsto \alpha \beta$ for $\beta \in \End^0(X)$ which is a $\bQ$-linear transformation on $\End^0(X)$.
Note that the above definition of $\chi_\alpha(t)$ makes no use of the fact that $\End^0(X)$ is semisimple.
Actually, for semisimple $\bQ$-algebras, it is much more useful to consider the so-called reduced characteristic polynomials.

We recall some basic definitions on semisimple algebras (see \cite[\S9]{Reiner03} for more details).

\begin{definition}
\label{def:red-char}
Let $R$ be a finite-dimensional semisimple algebra over a field $F$ with $\ch F = 0$, and write
\[
R = \bigoplus_{i=1}^{k} R_i,
\]
where each $R_i$ is a simple $F$-algebra.
For any element $r \in R$, as above, we denote by $\chi_r(t)$ the {\it characteristic polynomial of $r$}.
Namely, $\chi_r(t)$ is the characteristic polynomial of the left multiplication $r_L \colon r' \mapsto rr'$ for $r' \in R$.
Let $K_i$ be the center of $R_i$.
Then there exists a finite field extension $E_i/K_i$ splitting $R_i$ (cf.~\cite[\S7b]{Reiner03}), i.e., we have
\[
h_i \colon R_i \otimes_{K_i} E_i \xrightarrow{\ \ \sim\ \ } \Mat_{d_i}(E_i), \text{ where } [R_i:K_i] = d_i^2.
\]
Write $r = r_1 + \cdots + r_k$ with each $r_i \in R_i$.
We first define the {\it reduced characteristic polynomial $\chi_{r_i}^{\reduced}(t)$ of $r_i$} as follows (cf.~\cite[Definition~9.13]{Reiner03}):
\[
\chi_{r_i}^{\reduced}(t) \coloneqq \Nm_{K_i/F} \big(\det(t \, \bI_{d_i} - h_i(r_i \otimes_{K_i} \! 1_{E_i})) \big) \in F[t].
\]
It turns out that $\det(t \, \bI_{d_i} - h_i(r_i \otimes_{K_i} \! 1_{E_i}))$ lies in $K_i[t]$, and is independent of the choice of the splitting field $E_i$ of $R_i$ (cf.~\cite[Theorem~9.3]{Reiner03}).
The {\it reduced norm of $r_i$} is defined by
\[
\Nm_{R_i/F}^{\reduced}(r_i) \coloneqq \Nm_{K_i/F} \big(\det (h_i(r_i \otimes_{K_i} \! 1_{E_i})) \big) \in F.
\]
Finally, as one expects, the {\it reduced characteristic polynomial $\chi_{r}^{\reduced}(t)$} and the {\it reduced norm $\Nm_{R/F}^{\reduced}(r)$ of $r$} are defined by the products:
\begin{equation*}
\label{eq:def-red-char}
\chi_{r}^{\reduced}(t) \coloneqq \prod_{i=1}^{k} \chi_{r_i}^{\reduced}(t) \ \text{ and } \ 
\Nm_{R/F}^{\reduced}(r) \coloneqq \prod_{i=1}^{k} \Nm_{R_i/F}^{\reduced}(r_i).
\end{equation*}
\end{definition}

\begin{remark}
\label{rmk:red-char}
\begin{enumerate}[(1)]
\item It follows from \cite[Theorem~9.14]{Reiner03} that
\begin{equation}
\label{eq:Rei03-Thm9.14}
\chi_r(t) = \prod_{i=1}^{k} \chi_{r_i}(t) = \prod_{i=1}^{k} \chi_{r_i}^{\reduced}(t)^{d_i}.
\end{equation}
\item Note that reduced characteristic polynomials and norms are not affected by change of ground field (cf.~\cite[Theorem~9.27]{Reiner03}).
\end{enumerate}
\end{remark}

We now apply the above algebraic setting to $R = \End^0(X)$.
For any $\alpha \in \End(X)$, let $\chi_{\alpha}^{\reduced}(t)$ denote the reduced characteristic polynomial of $\alpha$ as an element of the semisimple $\bQ$-algebra $\End^0(X)$.
For simplicity, let us first consider the case when $X=A$ is a simple abelian variety and hence $D\coloneqq \End^0(A)$ is a division ring.
Let $K$ denote the center of $D$ which is a field, and $K_0$ the maximal totally real subfield of $K$. Set
\[
d^2 = [D:K], \ e = [K:\bQ] \ \text{ and } \ e_0 = [K_0:\bQ].
\]
Then the equality \eqref{eq:Rei03-Thm9.14} reads as
\[
\chi_\alpha(t) = \chi_{\alpha}^{\reduced}(t)^d.
\]
The lemma below shows that the two polynomials $P_\alpha(t)$ and $\chi_{\alpha}(t)$ are closely related.
Its proof relies on a characterization of normal forms of $D$ over $\bQ$.

For convenience, we include the following definition.
Let $R$ be a finite-dimensional associative algebra over an infinite field $F$.
A {\it norm form} on $R$ over $F$ is a non-zero polynomial function
\[
N_{R/F} \colon R \lra F
\]
(i.e., in terms of a basis of $R$ over $F$, $N_{R/F}(r)$ can be written as a polynomial over $F$ in the components of $r$) such that $N_{R/F}(rr') = N_{R/F}(r) N_{R/F}(r')$ for all $r,r'\in R$.

\begin{lemma}
\label{lemma:red-char-I}
Using notation as above, for any $\alpha\in \End(A)$, we have
\[
P_\alpha(t) = \chi_{\alpha}^{\reduced}(t)^m,
\]
where $m = 2g/(ed)$ is a positive integer.
In particular, the two polynomials $P_\alpha(t)$ and $\chi_\alpha(t)$ have the same complex roots (apart from multiplicities).\footnote{I would like to thank Yuri Zarhin for showing me an argument using the canonical norm form to prove this Lemma~\ref{lemma:red-char-I}.}
\end{lemma}
\begin{proof}
By the lemma in \cite[\S19]{Mumford} (located between Corollary~3 and Theorem~4, p.~179), any norm form of $D$ over $\bQ$ is of the following type
\[
(\Nm_{K/\bQ} \circ \Nm^{\reduced}_{D/K})^k \colon D \lra \bQ
\]
for a suitable nonnegative integer $k$, where $\Nm^{\reduced}_{D/K}$ is the reduced norm (aka canonical norm form in the sense of Mumford) of $D$ over $K$.
Now for each $n\in \bZ$, we have
\[
\chi_{\alpha}^{\reduced}(n) = \Nm_{K/\bQ} \circ \Nm^{\reduced}_{D/K}(n_A - \alpha).
\]
On the other hand, the action of $D$ on $V_{\ell} A \coloneqq T_{\ell} A \otimes_{\bZ_{\ell}} \bQ_{\ell}$ defines the determinant map
\[
\det \colon D \lra \bQ_{\ell},
\]
which actually takes on values in $\bQ$ and is a norm form of degree $2g$.
Indeed, let $V_\ell \alpha$ denote the induced map of $\alpha$ on $V_\ell A$, then $P_\alpha(n) = \deg(n_A - \alpha) = \det(n_A - \alpha) = \det(n - V_\ell \alpha)$ for all integers $n$ (see \cref{thm:char-poly}).
Applying the aforementioned lemma in \cite[\S19]{Mumford} to this $\det$, we obtain that for a suitable $m$,
\[
\det(\psi) = (\Nm_{K/\bQ}\circ \Nm^{\reduced}_{D/K}(\psi))^m
\]
for all $\psi \in D$.
It is easy to see that $m$ is $2g/(ed)$.
Then by taking $\psi = n_A - \alpha$, we have that $P_\alpha(n) = \chi_{\alpha}^{\reduced}(n)^m$ for all integers $n$.
This yields that $P_\alpha(t) = \chi_{\alpha}^{\reduced}(t)^m$.
\end{proof}

It is straightforward to generalize \cref{lemma:red-char-I} to the case that $X$ is the $n$-{th} power $A^n$ of a simple abelian variety $A$ since $\End^0(A^n) = \Mat_n(\End^0(A))$ is still a simple $\bQ$-algebra.

\begin{lemma}
\label{lemma:red-char-II}
Let $A$ be a simple abelian variety and $X = A^n$.
Let $\chi_{\alpha}^{\reduced}(t)$ denote the reduced characteristic polynomial of $\alpha$ as an element of the simple $\bQ$-algebra $\End^0(X) = \Mat_n(D)$ with $D = \End^0(A)$.
Then
\[
\chi_{\alpha}(t) = \chi_{\alpha}^{\reduced}(t)^{dn} \text{ and } P_\alpha(t) = \chi_{\alpha}^{\reduced}(t)^m,
\]
where $m = 2g/(edn)$ is a positive integer.
In particular, these two polynomials $P_\alpha(t)$ and $\chi_\alpha(t)$ have the same complex roots (apart from multiplicities).
\end{lemma}

We recall the following useful structure theorems on $\NS^0(X)$ which play a crucial role in the proof of our main theorem.

\begin{theorem}[{cf.~\cite[\S21, Application~III]{Mumford}}]
\label{thm:NS}
Fix a polarization $\phi \colon X \lra \what X$ that is an isogeny from $X$ to its dual $\what X$ induced from some ample line bundle $\sL_0$ (we suppress this $\sL_0$ since it does not make an appearance here henceforth).
Then the natural map
\[
\NS^0(X) \lra \End^0(X) \quad \text{via} \quad \sL \longmapsto \phi^{-1} \circ \phi_\sL
\]
is injective and its image is precisely the subspace $\big\{ \psi \in \End^0(X) \mid \psi^{\dagger} = \psi \big\}$ of symmetric elements of $\End^0(X)$ under the Rosati involution $^\dagger$ which maps $\psi$ to $\psi^{\dagger} \coloneqq \phi^{-1} \circ \what \psi \circ \phi$.
\end{theorem}

\begin{theorem}[{cf.~\cite[\S21, Theorems~2 and 6]{Mumford}}]
\label{thm:NS-matrix-form}
The endomorphism $\bR$-algebra $\End(X)_\bR \coloneqq \End^0(X) \otimes_\bQ \bR$ is isomorphic to a product of copies of $\Mat_r(\bR)$, $\Mat_r(\bC)$ and $\Mat_r(\bH)$.
Moreover, one can fix an isomorphism so that it carries the Rosati involution into the standard involution $\bA \longmapsto \ol\bA^\sT$.
In particular, $\NS(X)_\bR \coloneqq \NS^0(X) \otimes_\bQ \bR$ is isomorphic to a product of Jordan algebras of the following types:
\begin{align*}
\sH_r(\bR) &= r \times r \text{ symmetric real matrices,} \\
\sH_r(\bC) &= r \times r \text{ Hermitian complex matrices,} \\
\sH_r(\bH) &= r \times r \text{ Hermitian quaternionic matrices.}
\end{align*}
\end{theorem}




\section{Proof of Theorem~\ref{thmA}}
\label{section-proof}


\subsection{Some results on dynamical degrees}
\label{subsec-prelim-dyn-deg}

We first prepare some results used later to prove our main theorem.
Recall that in the complex dynamics, the dynamical degrees are bimeromorphic invariants of the dynamics system (see e.g. \cite[Theorem~4.2]{DS17}).
We have also shown the birational invariance of numerical dynamical degrees in arbitrary characteristic (cf.~\cite[Lemma~2.8]{Hu19}).
Below is a similar consideration which should be of interest in its own right.
Note, however, that we have not shown the birational invariance of cohomological dynamical degrees, which is actually one of the questions raised by Truong (see \cite[Question~5]{Truong1611}).

\begin{lemma}
\label{lemma:bir-inv}
Let $\pi \colon X \lra Y$ be a surjective morphism of smooth projective varieties defined over $\bk$.
Let $f$ (resp. $g$) be a surjective self-morphism of $X$ (resp. $Y$) such that $\pi \circ f = g \circ \pi$.
Then $\chi_i(f) \ge \chi_i(g)$ for any $0\le i\le 2 \dim Y$ and $\lambda_k(f) \ge \lambda_k(g)$ for any $0\le k\le \dim Y$.
\end{lemma}
\begin{proof}
We have the following commutative diagram of $\bQ_\ell$-vector spaces:
\[
\begin{tikzcd}
H^i_{\et}(Y, \bQ_\ell) \arrow{r}{\pi^*} \arrow{d}[swap]{g^*} & H^{i}_{\et}(X, \bQ_\ell) \arrow{d}{f^*} \\
H^i_{\et}(Y, \bQ_\ell) \arrow{r}{\pi^*} & H^{i}_{\et}(X, \bQ_\ell).
\end{tikzcd}
\]
The first part follows readily from \cite[Proposition~1.2.4]{Kleiman68} which asserts that the pullback map $\pi^*$ on $\ell$-adic \'etale cohomology is injective and hence $\pi^*H^i_{\et}(Y, \bQ_\ell)$ is an $f^*$-invariant subspace of $H^i_{\et}(X, \bQ_\ell)$.
The second part is similar; see also \cite[Lemma~2.8]{Hu19} for a stronger version.
\end{proof}

The following useful inequality was already noticed by Truong \cite{Truong1611}. We provide a proof for the sake of completeness.

\begin{lemma}
\label{lemma:dyn-deg-ineq}
Let $X$ be a smooth projective varieties defined over $\bk$, and $f$ a surjective self-morphism of $X$.
Then we have $\lambda_k(f) \le \chi_{2k}(f)$ for any $0\le k \le \dim X$.
\end{lemma}
\begin{proof}
Note that the $\ell$-adic \'etale cohomology $H^{\bullet}_{\et}(X, \bQ_\ell)$ is a Weil cohomology after the non-canonical choice of an isomorphism $\bZ_\ell(1) \isom \bZ_\ell$ (cf.~\cite[Example~1.2.5]{Kleiman68}).
So we have the following cycle map
\[
\gamma_X^k \colon \CH^k(X) \lra H^{2k}_{\et}(X, \bQ_\ell),
\]
where the $k$-th Chow group $\CH^k(X)$ of $X$ denotes the group of algebraic cycles of codimension $k$ modulo linear equivalence, i.e., $\CH^k(X) \coloneqq {Z}^k(X)/ \!\sim$.
Recall that a cycle $Z\in {Z}^k(X)$ is homologically equivalent to zero if $\gamma_X^k(Z) = 0$.
Also, it is well-known that homological equivalence $\sim_{\rm hom}$ is finer than numerical equivalence $\num$ (cf.~\cite[Proposition~1.2.3]{Kleiman68}).
Hence we have the following diagram of finite-dimensional $\bQ_\ell$-vector spaces (respecting the natural pullback action $f^*$ by the functoriality of the cycle map):
\begin{equation}
\label{eq:num-coh}
\begin{tikzcd}
(\CH^k(X)/ \!\sim_{\rm hom}) \otimes_\bZ \bQ_\ell \arrow[r, hook] \arrow[d, two heads] & H^{2k}_{\et}(X, \bQ_\ell) \\
(\CH^k(X)/ \!\num) \otimes_\bZ \bQ_\ell = N^k(X) \otimes_\bZ \bQ_\ell. & 
\end{tikzcd}
\end{equation}
Thus \cref{lemma:dyn-deg-ineq} follows.
\end{proof}

\begin{remark}
\label{rmk:div-equiv-relation}
When $k=1$, by a theorem of Matsusaka \cite{Matsusaka57}, homological equivalence coincides with numerical equivalence (in general, Grothendieck's standard conjecture $D$ predicts that they are equal for all $k$).
Furthermore, after tensoring with $\bQ$, both of them are also equivalent to algebraic equivalence $\approx$.
Namely, we have
\[
\NS(X)_\bQ = (\CH^1(X)/ \!\approx) \otimes_\bZ \bQ \isom (\CH^1(X)/ \!\sim_{\rm hom}) \otimes_\bZ \bQ \isom N^1(X) \otimes_\bZ \bQ.
\]
In particular, the cycle map $\gamma_X^1$ induces an injection
\[
N^1(X) \otimes_\bZ \bQ_\ell \longinjmap H^{2}_{\et}(X, \bQ_\ell).
\]
\end{remark}

\subsection{Extension of the pullback action to endomorphism algebras}
\label{subsec-extension}

For an endomorphism $\alpha$ of an abelian variety $X$, the following easy lemma sheds the light on the connection between the first numerical dynamical degree $\lambda_1(\alpha)$ of $\alpha$ and the induced action $\alpha^*$ on the endomorphism $\bQ$-algebra $\End^0(X)$,
while the latter is closely related to the matrix representation of $\alpha$ in $\End(X)_\bR$ or $\End(X)_\bC$ (see e.g. \cref{lemma:Kronecker}).

\begin{lemma}
\label{lemma:NS-End}
Fix a polarization $\phi \colon X \lra \what X$ as in \cref{thm:NS}. For any endomorphism $\alpha$ of $X$, we can extend the pullback action $\alpha^*$ on $\NS^0(X)$ to $\End^0(X)$ as follows:
\[
\alpha^* \colon \End^0(X) \lra \End^0(X) \quad \text{via} \quad \psi \longmapsto \alpha^*\psi \coloneqq \alpha^{\dagger} \circ \psi \circ \alpha.\footnote{Here by abuse of notation, we still denote this action by $\alpha^*$.
We would always write $\alpha^*|_{\End^0(X)}$ to emphasize the acting space in practice.}
\]
\end{lemma}
\begin{proof}
We shall identify $\NS^0(X) \ni \sL$ with the subspace of symmetric elements $\phi^{-1}\circ\phi_\sL$ of the endomorphism $\bQ$-algebra $\End^0(X)$ in virtue of \cref{thm:NS}.
Then the natural pullback action $\alpha^*$ on $\NS^0(X)$ could be reinterpreted in the following way:
\begin{align*}
\alpha^* \colon \NS^0(X) & \lra \NS^0(X) \\
\phi^{-1} \circ \phi_\sL & \longmapsto \phi^{-1} \circ \phi_{\alpha^* \! \sL}.
\end{align*}
Note that $\phi^{-1} \circ \phi_{\alpha^* \! \sL} = \phi^{-1} \circ \what \alpha \circ \phi_\sL \circ \alpha = \alpha^{\dagger} \circ \phi^{-1} \circ \phi_\sL \circ \alpha$,
where $\what \alpha$ is the induced dual endomorphism of $\what X$ and $\alpha^\dagger = \phi^{-1} \circ \what \alpha \circ \phi$ is the Rosati involution of $\alpha$;
for the first equality, see \cite[\S15, Theorem~1]{Mumford}.
This gives rise to an action of $\alpha$ on the whole endomorphism algebra $\End^0(X)$ by 
sending $\psi\in \End^0(X)$ to $\alpha^{\dagger} \circ \psi \circ \alpha$.
It is easy to see that the restriction of $\alpha^*|_{\End^0(X)}$ to $\NS^0(X)$ is just the natural pullback action $\alpha^*$ on $\NS^0(X)$.
\end{proof}

The lemma below plays a crucial role in the proof of our main theorem by giving a characterization of the above induced action $\alpha^*$ on certain endomorphism algebras of abelian varieties.
Here we consider a more general version from the aspect of linear algebra.

\begin{lemma}
\label{lemma:Kronecker}
\begin{enumerate}[{\em (1)}]
\item \label{lemma:Kronecker1}
If $\bA\in \Mat_n(\bR)$, then the linear transformation
\[
f_\bA \colon \Mat_n(\bR) \lra \Mat_n(\bR) \quad \text{via} \quad \bB \longmapsto \bA\!^\sT \bB \bA
\]
of $n^2$-dimensional $\bR$-vector space $\Mat_n(\bR)$ could be represented by $\bA\otimes \bA$, the Kronecker product of $\bA$ and itself.

\item \label{lemma:Kronecker2}
If $\bA\in \Mat_n(\bC)$, then the following linear transformation
\[
f_\bA \colon \Mat_n(\bC) \lra \Mat_n(\bC) \quad \text{via} \quad \bB \longmapsto \ol\bA^\sT \bB \bA
\]
of $n^2$-dimensional $\bC$-vector space $\Mat_n(\bC)$ could be represented by $\bA\otimes \ol\bA$, the Kronecker product of $\bA$ and its complex conjugate $\ol\bA$.

\item \label{lemma:Kronecker3}
If $\bA\in \Mat_n(\bC)$, then the following linear transformation
\[
f_\bA \colon \Mat_n(\bC) \lra \Mat_n(\bC) \quad \text{via} \quad \bB \longmapsto \ol\bA^\sT \bB \bA
\]
of $2n^2$-dimensional $\bR$-vector space $\Mat_n(\bC)$ could be represented by the block diagonal matrix $(\bA\otimes \ol\bA) \oplus (\ol\bA\otimes \bA)$.
\end{enumerate}
\end{lemma}
\begin{proof}
We first prove the assertion \eqref{lemma:Kronecker2} since the proof of the first one is essentially the same.
Choose the standard $\bC$-basis $\{\be_{ij}\}$ of $\Mat_n(\bC)$, where $\be_{ij}$ denotes the $n\times n$ complex matrix whose $(i,j)$-entry is $1$, and $0$ elsewhere.
We also adopt the standard vectorization
\[
\vect \colon \Mat_n(\bC) \xrightarrow{\ \ \sim\ \ } \bC^{n^2}
\]
of $\Mat_n(\bC)$, which converts $n\times n$ matrices into column vectors so that
\begin{equation}
\label{eq:basis}
\{ \vect(\be_{11}), \vect(\be_{21}), \dots, \vect(\be_{n1}), \vect(\be_{12}), \dots, \vect(\be_{n2}), \dots, \vect(\be_{1n}), \dots, \vect(\be_{nn}) \}
\end{equation}
forms the standard $\bC$-basis of $\bC^{n^2}$.
Write $\bA = (a_{ij})_{n\times n}$ with $a_{ij}\in \bC$.
Then we have
\[
\ol\bA^\sT \cdot \be_{ij} = \ol a_{i1} \be_{1j} + \ol a_{i2} \be_{2j} + \cdots + \ol a_{in} \be_{nj}.
\]
Hence under the basis \eqref{eq:basis}, it is easy to verify that the left multiplication by $\ol\bA^\sT$ on the $\bC$-vector space $\Mat_n(\bC) \isom \bC^{n^2}$ is represented by the block diagonal matrix $\ol\bA \oplus \ol\bA \oplus \cdots \oplus \ol\bA = \bI_n \otimes \ol\bA$.
Similarly, since $\be_{ij} \cdot \bA = a_{j1} \be_{i1} + a_{j2} \be_{i2} + \cdots + a_{jn} \be_{in}$, one can check that under the basis \eqref{eq:basis}, the right multiplication by $\bA$ is represented by $\bA \otimes \bI_n$.
Therefore, our linear map $f_\bA$ is represented by the matrix product $(\bI_n \otimes \ol\bA) \cdot (\bA \otimes \bI_n) = \bA \otimes \ol\bA$.
Thus the assertion \eqref{lemma:Kronecker2} follows.

For the last assertion, we just need to combine the assertion \eqref{lemma:Kronecker2} with the following general fact: if $\bM \in \Mat_n(\bC)$, then the associated $2n\times 2n$ real matrix
\[
\begin{pmatrix}
\re \bM & -\im\bM \\
\im \bM & \re \bM
\end{pmatrix}
\]
is similar to the block diagonal matrix $\bM \oplus \ol\bM$.
Indeed, one can easily verify that
\[
\begin{pmatrix}
\bI_n & -i \, \bI_n \\
-i \, \bI_n & \bI_n
\end{pmatrix}^{-1}
\cdot
\begin{pmatrix}
\re \bM & -\im\bM \\
\im \bM & \re \bM
\end{pmatrix}
\cdot
\begin{pmatrix}
\bI_n & -i \, \bI_n \\
-i \, \bI_n & \bI_n
\end{pmatrix}
=
\begin{pmatrix}
\bM & \mathbf{0} \\
\mathbf{0} & \ol\bM
\end{pmatrix}.
\]
Applying the above fact to the complex matrix $\bA\otimes \ol\bA$ coming from the assertion \eqref{lemma:Kronecker2}, one gets the assertion \eqref{lemma:Kronecker3} and hence \cref{lemma:Kronecker} follows.
\end{proof}

\subsection{Several standard reductions towards the proof}
\label{subsec-reductions}

Before proving our main \cref{thmA}, we start with some standard reductions.
The lemma below reduces the general case to the splitting product case.

\begin{lemma}
\label{lemma:reduction-product}
In order to prove \cref{thmA}, it suffices to consider the following case:
\begin{itemize}
\item the abelian variety $X = A_1^{n_1} \times \cdots \times A_s^{n_s}$, where the $A_j$ are mutually non-isogenous simple abelian varieties, and
\item the surjective self-morphism $f$ of $X$ is a surjective endomorphism $\alpha$ which can be written as $\alpha_1 \times \cdots \times \alpha_s$ with $\alpha_j \in \End(A_j^{n_j})$.
\end{itemize}
\end{lemma}
\begin{proof}
We claim that it suffices to consider the case when $f = \alpha$ is a surjective endomorphism.
Indeed, any morphism (i.e., regular map) of abelian varieties is a composite of a homomorphism with a translation (cf.~\cite[Corollary~2.2]{Milne86}).
Hence we can write $f$ as $t_x \circ \alpha$ for a surjective endomorphism $\alpha\in \End(X)$ and $x\in X(\bk)$.
Note however that $t_x\in \Aut^0(X) \isom X$ acts as identity on $H^1_{\et}(X, \bQ_\ell)$ and hence on $H^i_{\et}(X, \bQ_\ell)$ for all $i$.
It follows from the functoriality of the pullback map on $\ell$-adic \'etale cohomology that $\chi_i(f) = \chi_i(\alpha)$.
Similarly, we also get $\lambda_k(f) = \lambda_k(\alpha)$ for all $k$.
So the claim follows, and from now on our $f = \alpha$ is an isogeny.

We then make another claim as follows.
\begin{claim}
\label{claim:reduction}
Towards the proof of \cref{thmA}, we are free to replace our pair $(X, \alpha)$ by any of the following pairs:
\begin{enumerate}
\item \label{claim:reduction1}
$(X, \alpha^m)$, for any positive integer $m$;

\item \label{claim:reduction2}
$(X, m\alpha)$, for any positive integer $m$;

\item \label{claim:reduction3}
$(X', \alpha' \coloneqq g \circ \alpha \circ h)$,
where $g \colon X \to X'$ and $h \colon X' \to X$ are isogenies such that $h\circ g = m_X$ and $g \circ h = m_{X'}$ with $m = \deg g$.
\end{enumerate}
\end{claim}
\begin{proof}[Proof of \cref{claim:reduction}] \renewcommand{\qedsymbol}{}
The first part follows from the functoriality of the pullback map.
For the second one, we note that $m\alpha = m_X \circ \alpha = \alpha \circ m_X$, where $m_X$ is the multiplication by $m$ map.
Using the isomorphism $H^1_{\et}(X, \bZ_\ell) \isom \Hom_{\bZ_\ell}(T_\ell X, \bZ_\ell)$,
one can easily see that the induced pullback map $m_X^*$ on $H^1_{\et}(X, \bQ_\ell)$ is also the multiplication by $m$ map, and hence $m_X^*|_{H^i_{\et}(X, \bQ_\ell)}$ is represented by the diagonal matrix $m^i \cdot \id_{H^i_{\et}(X, \bQ_\ell)}$; see e.g. \cref{thm:etale-coh}.
It follows from the diagram \eqref{eq:num-coh} in the proof of \cref{lemma:dyn-deg-ineq} that the pullback map $m_X^*$ on each $N^k(X)_\bR$ is also represented by the diagonal matrix $m^{2k} \cdot \id_{N^k(X)_\bR}$.
In particular, we have $\chi_i(m\alpha) = m^i \chi_i(\alpha)$ and $\lambda_k(m\alpha) = m^{2k} \lambda_k(\alpha)$, which yields the part \eqref{claim:reduction2}.

For the last part, it is easy to verify that $\alpha' \circ g = g \circ (m\alpha)$ and $h \circ \alpha' = (m\alpha) \circ h$.
By applying \cref{lemma:bir-inv} to the isogenies $g$ and $h$, we have $\chi_i(\alpha') = \chi_i(m\alpha)$ and $\lambda_k(\alpha') = \lambda_k(m\alpha)$.
Then combining with the second part, the third one follows.
So we have proved \cref{claim:reduction}.
\end{proof}

Let us go back to the proof of \cref{lemma:reduction-product}.
By Poincar\'e's complete reducibility theorem (cf.~\cite[\S19, Theorem~1]{Mumford}), we know that $X$ is isogenous to the product $A_1^{n_1} \times \cdots \times A_s^{n_s}$, where the $A_j$ are mutually non-isogenous simple abelian varieties.
Then
\[
\End^0(X) \isom \bigoplus_{j=1}^s \End^0(A_j^{n_j}),
\]
so that we can write $\alpha$ as $\alpha_1 \times \cdots \times \alpha_s$ with $\alpha_j \in \End^0(A_j^{n_j})$.
Using the reductions \eqref{claim:reduction2} and \eqref{claim:reduction3} in \cref{claim:reduction}, we only need to consider the case when $X$ itself is the product variety and each $\alpha_j$ belongs to $\End(A_j^{n_j})$, as stated in the lemma.
\end{proof}

\begin{remark}
\label{rmk:reduction-simple}
We are keen to further reduce the situation of \cref{lemma:reduction-product} to the case when $X = A^n$ is a power of some simple abelian variety $A$, as Esnault and Srinivas did in the proof of \cite[Proposition~6.2]{ES13}.
However, to the best of our knowledge, it does not seem to be straightforward.
More precisely, let $X$ and $\alpha$ be as in \cref{lemma:reduction-product}.
Suppose that \cref{thmA} holds for every $A_j^{n_j}$ and surjective endomorphism $\alpha_j \in \End(A_j^{n_j})$, i.e., $\lambda_1(\alpha_j) = \chi_2(\alpha_j)$ for all $j$.
We wish to show that \cref{thmA} also holds for $X$ and $\alpha$.
Note that
\[
\NS(X) \isom \bigoplus_{j=1}^s \NS(A_j^{n_j}).\footnote{In general, one has $\NS(X \times_\bk Y) \isom \NS(X) \oplus \NS(Y) \oplus \Hom_\bk(\Alb(X), \Pic^0(Y))$; see e.g. \cite[the proof of Theorem~3]{Tate66}. See also \cite[\S3.2]{BC16-Crelle} and references therein for more details about the divisorial correspondences.}
\]
It follows that
\begin{equation}
\label{eq:reduction-simple-1}
\lambda_1(\alpha) = \max_j \{ \lambda_1(\alpha_j)\} = \max_j \{ \chi_2(\alpha_j)\}.
\end{equation}
On the other hand, by the K\"unneth formula, we have
\begin{equation*}
\begin{gathered}
H^1_{\et}(X, \bQ_\ell) \isom \bigoplus_j H^1_{\et}(A_j^{n_j}, \bQ_\ell), \text{ and } \\    
H^2_{\et}(X, \bQ_\ell) \isom \bigoplus_j H^2_{\et}(A_j^{n_j}, \bQ_\ell) \bigoplus \bigoplus_{j < k} \big(H^1_{\et}(A_j^{n_j}, \bQ_\ell) \otimes H^1_{\et}(A_k^{n_k}, \bQ_\ell) \big).
\end{gathered}
\end{equation*}
However, we are not able to deduce that $\chi_2(\alpha) = \max_j \{ \chi_2(\alpha_j)\}$ due to the appearance of the tensor product of the $H^1_{\et}$.

For the sake of completeness, let us explain this obstruction in a more precise way.
We denote by $P_{\alpha_j}(t) \in \bZ[t]$ the characteristic polynomial of $\alpha_j$ (or equivalently $T_\ell \alpha_j$, by \cref{thm:char-poly}).
Set $g_j = \dim A_j^{n_j}$.
Denote all complex roots of $P_{\alpha_j}(t)$ by $\omega_{j,1}, \ldots, \omega_{j,2g_j}$.
Without loss of generality, we may assume that
\begin{equation}
\label{eq:reduction-simple-2}
|\omega_{j,1}| \ge \cdots \ge |\omega_{j,2g_j}| \text{ for all $1\le j \le s$, and }
|\omega_{1,1}| \ge \cdots \ge |\omega_{s,1}|.
\end{equation}
It follows from \cref{thm:etale-coh} that $\chi_2(\alpha_j) = |\omega_{j,1}| \cdot |\omega_{j,2}|$ for all $j$.
Suppose that
\begin{equation}
\label{eq:reduction-simple-3}
\max_j \{ \chi_2(\alpha_j)\} = \chi_2(\alpha_{j_0}) = |\omega_{j_0,1}| \cdot |\omega_{j_0,2}| \text{ for some $j_0$}.
\end{equation}
Note that $j_0$ may not be $1$.
If $|\omega_{2,1}| \le |\omega_{1,2}|$ (in particular, $j_0$ is $1$), then
\[
\chi_2(\alpha) = |\omega_{1,1}| \cdot |\omega_{1,2}| = \chi_2(\alpha_1) = \max_j \{ \chi_2(\alpha_j)\} = \lambda_1(\alpha).
\]
So we are done in this case.
However, if $|\omega_{2,1}| > |\omega_{1,2}|$, then
\[
\chi_2(\alpha) = |\omega_{1,1}| \cdot |\omega_{2,1}| \ge |\omega_{j_0,1}| \cdot |\omega_{j_0,2}| = \chi_2(\alpha_{j_0}) = \max_j \{ \chi_2(\alpha_j)\} = \lambda_1(\alpha).
\]
There is no obvious reason to exclude the worst case $j_0=1$ which yields that
\[
\chi_2(\alpha) = |\omega_{1,1}| \cdot |\omega_{2,1}| > |\omega_{1,1}| \cdot |\omega_{1,2}| = \chi_2(\alpha_1) = \max_j \{ \chi_2(\alpha_j)\} = \lambda_1(\alpha).
\]
\end{remark}

To proceed, we observe that over complex number field $\bC$, the above pathology does not happen because each eigenvalue $\omega_{j,2}$ turns out to be the complex conjugate of $\omega_{j,1}$. This fact follows from the Hodge decomposition $H^1(X, \bC) = H^{1,0}(X) \oplus \ol{H^{1,0}(X)}$, which does not seem to exist in \'etale cohomology as far as we know.
But we still believe that $\omega_{j,2} = \ol{\omega_{j,1}}$ for all $j$.
(As a consequence of our main theorem, we will see that this is actually true; see \cref{final-remark}.)
The following lemma makes use of this observation to reduce the splitting product case as in \cref{lemma:reduction-product} to the case when $X = A^n$ for some simple abelian variety $A$.

\begin{lemma}
\label{lemma:reduction-simple}
In order to prove \cref{thmA}, it suffices to show that if $A^n$ is a power of a simple abelian variety $A$ and $\alpha \in \End(A^n)$ is a surjective endomorphism of $A^n$, then $\lambda_1(\alpha) = |\omega_1|^2$, where $\omega_1$ is one of the complex roots of the characteristic polynomial $P_\alpha(t)$ of $\alpha$ with the maximal absolute value.
\end{lemma}
\begin{proof}
Thanks to \cref{lemma:reduction-product}, let us consider the case when the abelian variety $X = A_1^{n_1} \times \cdots \times A_s^{n_s}$, where the $A_j$ are mutually non-isogenous simple abelian varieties, and $\alpha = \alpha_1 \times \cdots \times \alpha_s$ is a surjective endomorphism of $X$ with $\alpha_j \in \End(A_j^{n_j})$.
We assume that the reader has been familiar with the notation introduced in \cref{rmk:reduction-simple}, in particular, \cref{eq:reduction-simple-1,eq:reduction-simple-2,eq:reduction-simple-3}.
Applying the hypothesis of \cref{lemma:reduction-simple} to each $A_j^{n_j}$ and $\alpha_j$, we have $\lambda_1(\alpha_j) = |\omega_{j,1}|^2$.
It follows from \cref{lemma:dyn-deg-ineq} and \cref{thm:etale-coh} that
$\lambda_1(\alpha_j) \le \chi_2(\alpha_j) = |\omega_{j,1}| \cdot |\omega_{j,2}|$.
Hence $\lambda_1(\alpha_j) = \chi_2(\alpha_j)$ and $|\omega_{j,1}| = |\omega_{j,2}|$ for all $j$ which tells us $j_0=1$.
This yields that
\[
\chi_2(\alpha) = |\omega_{1,1}| \cdot |\omega_{1,2}| = \chi_2(\alpha_1) = \max_j \{ \chi_2(\alpha_j)\} = \max_j \{ \lambda_1(\alpha_j)\} = \lambda_1(\alpha).
\]
The first and second equalities follow again from \cref{thm:etale-coh}, the third one holds because $j_0=1$, \cref{eq:reduction-simple-1} gives the last one.
\end{proof}

\subsection{Proof of Theorem~\ref{thmA}}
\label{subsec-proof}

We are now ready to prove the main theorem.

\begin{proof}[Proof of \cref{thmA}]
By \cref{lemma:reduction-simple}, we can assume that $X = A^n$ for some simple abelian variety $A$ and $\alpha \in \End(X)$ is a surjective endomorphism of $X$.
Let $P_\alpha(t) \in \bZ[t]$ be the characteristic polynomial of $\alpha$ (see \cref{thm:char-poly}).
Set $g = \dim X$.
Denote all complex roots of $P_\alpha(t)$ by $\omega_1, \ldots, \omega_{2g}$.
Without loss of generality, we may assume that
\[
|\omega_1| \ge \cdots \ge |\omega_{2g}|.
\]
We shall prove that
\begin{equation}
\label{eq:key}
\lambda_1(\alpha) = |\omega_1|^2,
\end{equation}
which will conclude the proof of the theorem by \cref{lemma:reduction-simple}.

Under the above assumption, the endomorphism algebra $\End^0(X)$ is the simple $\bQ$-algebra $\Mat_n(D)$ of all $n\times n$ matrices with entries in the division ring $D \coloneqq \End^0(A)$.
Let $K$ denote the center of $D$, and $K_0$ the maximal totally real subfield of $K$.
As usual, we set
\[
d^2 = [D:K], \ e = [K:\bQ] \ \text{ and } \ e_0 = [K_0:\bQ].
\]
Note that by \cref{lemma:NS-End}, the natural pullback action $\alpha^*$ on $\NS^0$ can be extended to an action $\alpha^*$ on the whole endomorphism $\bQ$-algebra $\End^0(X)$ as follows:
\[
\alpha^* \colon \End^0(X) \lra \End^0(X) \quad \text{via} \quad \psi \longmapsto \alpha^{\dagger} \circ \psi \circ \alpha.
\]
On the other hand, by tensoring with $\bR$, we know that
\[
\End(X)_\bR = \End^0(X) \otimes_\bQ \bR \isom \Mat_n(D) \otimes_\bQ \bR \isom \Mat_n(D \otimes_\bQ \bR)
\]
is either a product of $\Mat_r(\bR)$, $\Mat_r(\bC)$ or $\Mat_r(\bH)$ with $\NS(X)_\bR$ being a product of $\sH_r(\bR)$, $\sH_r(\bC)$ or $\sH_r(\bH)$, the corresponding subspace of symmetric/Hermitian matrices (see \cref{thm:NS-matrix-form}).
When there is no risk of confusion, for simplicity, we still denote the induced action $\alpha^* \otimes_\bQ 1_\bR$ by $\alpha^*$.
In particular, we would write $\alpha^*|_{\End(X)_\bR}$ and $\alpha^*|_{\NS(X)_\bR}$ to emphasize the acting spaces.

According to Albert's classification of the endomorphism $\bQ$-algebra $D$ of a simple abelian variety $A$ (cf.~\cite[\S21, Theorem~2]{Mumford}), we have the following four cases.

\begin{case}
\label{case-I}
$D$ is of Type I$(e)$: $d=1$, $e=e_0$ and $D=K=K_0$ is a totally real algebraic number field and the involution (on $D$) is the identity. In this case, 
\[
\End(X)_\bR \isom \bigoplus_{i=1}^{e_0} \Mat_n(\bR) \ \text{ and } \ \NS(X)_\bR \isom \bigoplus_{i=1}^{e_0} \sH_n(\bR).
\]
For our $\alpha \in \End(X)$, let us denote its image $\alpha \otimes_\bZ 1_\bR$ in $\End(X)_\bR$ by the block diagonal matrix
$
\bA_\alpha = \bA_{\alpha,1} \oplus \cdots \oplus \bA_{\alpha,e_0}
$
with each $\bA_{\alpha,i} \in \Mat_n(\bR)$.
Then the Rosati involution $\alpha^\dagger$ of $\alpha$ could be represented by the transpose $\bA_\alpha^\sT = \bA_{\alpha,1}^\sT \oplus \cdots \oplus \bA_{\alpha,e_0}^\sT$ (see \cref{thm:NS-matrix-form}).
Hence we can rewrite the induced action $\alpha^*$ on $\End(X)_\bR$ in the following matrix form:
\[
\bB = \bB_{1} \oplus \cdots \oplus \bB_{e_0} \longmapsto \bA_\alpha^\sT \bB \bA_\alpha = \bA_{\alpha,1}^\sT \bB_1 \bA_{\alpha,1} \oplus \cdots \oplus \bA_{\alpha,e_0}^\sT \bB_{e_0} \bA_{\alpha,e_0}.
\]
Thanks to \cref{lemma:Kronecker}~\eqref{lemma:Kronecker1}, for each $i$, the linear transformation defined by the mapping
\[
\bB_i \in \Mat_n(\bR) \mapsto \bA_{\alpha,i}^\sT \bB_i \bA_{\alpha,i} \in \Mat_n(\bR),
\]
can be represented by the Kronecker product $\bA_{\alpha,i} \otimes \bA_{\alpha,i}$.
Hence the above linear transformation $\alpha^* |_{\End(X)_\bR}$ on the $e_0 n^2$-dimensional $\bR$-vector space $\End(X)_\bR$ is represented by the block diagonal matrix
\[
(\bA_{\alpha,1} \otimes \bA_{\alpha,1}) \oplus \cdots \oplus (\bA_{\alpha,e_0} \otimes \bA_{\alpha,e_0}).
\]

For each $1\le i\le e_0$, denote all eigenvalues of $\bA_{\alpha,i}$ by $\pi_{i,1}, \ldots, \pi_{i,n}$.
It thus follows from the above discussion that all eigenvalues of the linear transformation $\alpha^* |_{\End(X)_\bR}$ are exactly $\pi_{i,j}\pi_{i,k}$ with $1\le j, k\le n$ and $1\le i\le e_0$.
In particular, if $\bv_{i,j}$ and $\bv_{i,k}$ denote eigenvectors of $\bA_{\alpha,i}$ corresponding to $\pi_{i,j}$ and $\pi_{i,k}$, respectively, then
\[
\bv_{i,j} \otimes \bv_{i,k} = \vect(\bv_{i,j}^\sT \otimes \bv_{i,k}) = \vect(\bv_{i,k} \otimes \bv_{i,j}^\sT) = \vect(\bv_{i,k} \cdot \bv_{i,j}^\sT)
\]
is the eigenvector of $\bA_{\alpha,i} \otimes \bA_{\alpha,i}$ corresponding to $\pi_{i,j}\pi_{i,k}$.\footnote{Note that due to multiplicities of eigenvalues, $\bA_{\alpha,i}$ does not necessarily have $n$ distinct eigenvalues.
Thus, $\bv_{i,j}$ and $\bv_{i,k}$ may be the same for different $j$ and $k$.
Also, not all eigenvectors of $\bA_{\alpha,i} \otimes \bA_{\alpha,i}$ have to arise in this way, namely, being the tensor products $\bv_{i,j} \otimes \bv_{i,k}$.
For instance, one could consider a Jordan block $J_{\lambda,2} \in \Mat_2(\bR)$ with the eigenvalue $\lambda$, but $J_{\lambda,2} \otimes J_{\lambda,2} \sim J_{\lambda^2,1} \oplus J_{\lambda^2,3}$.
}
Now, according to \cref{rmk:red-char}, the reduced characteristic polynomial $\chi_{\alpha}^{\reduced}(t)$ of $\alpha$ is independent of the change of the ground field,
and hence equal to the reduced characteristic polynomial $\chi_{\alpha \otimes_\bZ 1_\bR}^{\reduced}(t)$ of $\alpha \otimes_\bZ 1_\bR \in \End(X)_\bR$,
while the latter by \cref{def:red-char} is just the characteristic polynomial $\det(t \, \bI_{e_0n} - \bA_\alpha)$ of $\bA_\alpha$.
Hence, without loss of generality, we may assume that $\omega_1 = \pi_{1,1}$ by \cref{lemma:red-char-II}.

We now have two subcases to consider.
If $\pi_{1,1} \in \bR$ so that $\bv_{1,1}$ is also a real eigenvector,
then $\bv_{1,1} \otimes \bv_{1,1}$ is a real eigenvector of $\alpha^* |_{\End(X)_\bR}$ corresponding to the eigenvalue $\pi_{1,1}^2$.
The associated column vector of this eigenvector is the real symmetric matrix $\bv_{1,1} \otimes \bv_{1,1}^\sT = \bv_{1,1}^\sT \otimes \bv_{1,1}$.
Next, let us assume that $\pi_{1,1} \in \bC \setminus \bR$.
Then $\ol \pi_{1,1}$ is another eigenvalue of $\bA_{\alpha,1}$ with the corresponding eigenvector $\ol\bv_{1,1}$, since $\bA_{\alpha,1}$ is defined over $\bR$.
It follows that $\bv_{1,1} \otimes \ol\bv_{1,1} + \ol\bv_{1,1} \otimes \bv_{1,1}$ is a real eigenvector of $\alpha^* |_{\End(X)_\bR}$ corresponding to the eigenvalue $\pi_{1,1} \ol\pi_{1,1} = |\pi_{1,1}|^2$;
moreover, it is the associated column vector of the real symmetric matrix
\[
\bv_{1,1}^\sT \otimes \ol\bv_{1,1} + \ol\bv_{1,1}^\sT \otimes \bv_{1,1} = \ol\bv_{1,1} \otimes \bv_{1,1}^\sT + \ol\bv_{1,1}^\sT \otimes \bv_{1,1}.
\]
In either case, we have shown that the spectral radii of $\alpha^* |_{\End(X)_\bR}$ and $\alpha^* |_{\NS(X)_\bR}$ coincide, both equal to $|\pi_{1,1}|^2$.
In summary, we have
\[
|\omega_1|^2 = |\pi_{1,1}|^2 = \rho(\alpha^* |_{\End(X)_\bR}) = \rho(\alpha^* |_{\NS(X)_\bR}) = \lambda_1(\alpha).
\]
For the last equality, see \cref{rmk:div-equiv-relation}.
So we conclude the proof of the equality \eqref{eq:key} in this case.
\end{case}

\begin{case}
\label{case-II}
$D$ is of Type II$(e)$: $d=2$, $e=e_0$, $K=K_0$ is a totally real algebraic number field and $D$ is an indefinite quaternion division algebra over $K$. Hence
\[
\End(X)_\bR \isom \bigoplus_{i=1}^{e_0} \Mat_{2n}(\bR) \ \text{ and } \ \NS(X)_\bR \isom \bigoplus_{i=1}^{e_0} \sH_{2n}(\bR).
\]
The rest is exactly the same as \cref{case-I}.
\end{case}

\begin{case}
\label{case-III}
$D$ is of Type III$(e)$: $d=2$, $e=e_0$, $K=K_0$ is a totally real algebraic number field and $D$ is a definite quaternion division algebra over $K$. In this case,
\[
\End(X)_\bR \isom \bigoplus_{i=1}^{e_0} \Mat_n(\bH) \ \text{ and } \ \NS(X)_\bR \isom \bigoplus_{i=1}^{e_0} \sH_n(\bH),
\]
where $\bH = \big(\frac{-1, \, -1}{\bR} \big)$ is the standard quaternion algebra over $\bR$.
Clearly, $\bH$ can be embedded, in a standard way (see e.g. \cite[Example~9.4]{Reiner03}), into $\Mat_2(\bC) \isom \bH \otimes_\bR \bC$.
This induces a natural embedding of $\Mat_n(\bH)$ into $\Mat_{2n}(\bC) \isom \Mat_n(\bH) \otimes_\bR \bC$ as follows (cf.~\cite[\S4]{Lee49}):
\[
\iota \colon \Mat_n(\bH) \longinjmap \Mat_{2n}(\bC) \quad \text{via} \quad \bA = \bA_1 + \bA_2 \, \mathbf{j} \longmapsto \iota (\bA) \coloneqq 
\begin{pmatrix}
\bA_1 & \bA_2 \\
-\overline{\bA}_2 & \overline{\bA}_1
\end{pmatrix}.
\]
In particular, a quaternionic matrix $\bA$ is Hermitian if and only if its image $\iota(\bA)$ is a Hermitian complex matrix.

For brevity, we only consider the case $e_0 = 1$ (to deal with the general case, the only cost is to introduce an index $i$ as we have done in \cref{case-I} since the matrices involved are block diagonal matrices).
Denote the image $\alpha \otimes_\bZ 1_\bR$ of $\alpha$ in $\Mat_{n}(\bH)$ by $\bA_\alpha = \bA_1 + \bA_2 \, \mathbf{j}$ with $\bA_1, \bA_2 \in \Mat_n(\bC)$.
Then the Rosati involution $\alpha^\dagger$ of $\alpha$ could be represented by the quaternionic conjugate transpose $\bA_\alpha^* = \ol\bA_\alpha^\sT$ (see \cref{thm:NS-matrix-form}), whose image under $\iota$ is just the complex conjugate transpose $\iota(\bA_\alpha)^*$ (aka Hermitian transpose) of $\iota(\bA_\alpha)$.
Similar as in \cref{lemma:NS-End}, the action $\alpha^*$ on $\End(X)_\bR \isom \Mat_n(\bH)$ can be extended to $\End(X)_\bC \coloneqq \End(X)_\bR \otimes_\bR \bC \isom \Mat_{2n}(\bC)$.
By abuse of notation, we still denote this induced action by $\alpha^* \colon \Mat_{2n}(\bC) \lra \Mat_{2n}(\bC)$, which maps $\bB$ to $\iota(\bA_\alpha)^* \cdot \bB \cdot \iota(\bA_\alpha)$.
It follows from \cref{lemma:Kronecker}~\eqref{lemma:Kronecker2} that $\alpha^*|_{\Mat_{2n}(\bC)}$ could be represented by the Kronecker product $\iota(\bA_\alpha) \otimes \ol{\iota(\bA_\alpha)}$.

Note that our $\End(X)_\bC \isom \Mat_{2n}(\bC)$ is a central simple $\bC$-algebra.
Then by \cref{def:red-char,rmk:red-char}, the reduced characteristic polynomial $\chi_{\alpha}^{\reduced}(t)$ of $\alpha$ is equal to the characteristic polynomial $\det(t \, \bI_{2n} - \iota(\bA_\alpha))$ of the complex matrix $\iota(\bA_\alpha)$.
Thanks to \cite[Theorem~5]{Lee49}, the $2n$ eigenvalues of $\iota(\bA_\alpha)$ fall into $n$ pairs, each pair consisting of two conjugate complex numbers;
denote them by $\pi_1, \ldots, \pi_{n}, \pi_{n+1} = \overline{\pi}_1, \ldots, \pi_{2n} = \overline{\pi}_{n}$.
In fact, it is easy to verify that if $\pi_i \in \bC$ is an eigenvalue of $\iota(\bA_\alpha)$ so that
\[
\iota(\bA_\alpha) \begin{pmatrix}
\bu_i \\
\bv_i 
\end{pmatrix}
= \pi_i \begin{pmatrix}
\bu_i \\
\bv_i 
\end{pmatrix},
\ \text{ then } \ 
\iota(\bA_\alpha) \begin{pmatrix}
-\ol\bv_i \\
\ol\bu_i 
\end{pmatrix}
= \ol\pi_i \begin{pmatrix}
-\ol\bv_i \\
\ol\bu_i 
\end{pmatrix},
\]
i.e., $\ol\pi_i$ is also an eigenvalue of $\iota(\bA_\alpha)$ corresponding to the eigenvector $(-\ol\bv_i^\sT, \ol\bu_i^\sT)^\sT$.
Therefore, without loss of generality, we may assume that $\omega_1 = \pi_1$ by \cref{lemma:red-char-II}.

Let $(\bu_1^\sT, \bv_1^\sT)^\sT$ denote an eigenvector of $\iota(\bA_\alpha)$ corresponding to the eigenvalue $\pi_1$.
Then $(-\ol\bv_1^\sT, \ol\bu_1^\sT)^\sT$ is an eigenvector of $\iota(\bA_\alpha)$ corresponding to the eigenvalue $\ol\pi_1$.
Since the linear transformation $\alpha^* |_{\End(X)_\bC}$ can be represented by $\iota(\bA_\alpha) \otimes \ol{\iota(\bA_\alpha)}$ (cf.~\cref{lemma:Kronecker}~\eqref{lemma:Kronecker2}),
we see that both $(\bu_1^\sT, \bv_1^\sT)^\sT \otimes (\ol\bu_1^\sT, \ol\bv_1^\sT)^\sT$ and $(-\ol\bv_1^\sT, \ol\bu_1^\sT)^\sT \otimes (-\bv_1^\sT, \bu_1^\sT)^\sT$ are eigenvectors of $\alpha^* |_{\End(X)_\bC}$, corresponding to the same eigenvalue $\pi_1\ol\pi_1$.
Recall that these two eigenvectors are the associated column vectors of the Hermitian complex matrices
\[
\begin{pmatrix}
\ol\bu_1 \\
\ol\bv_1 
\end{pmatrix}
\otimes
(\bu_1^\sT, \bv_1^\sT) =
\begin{pmatrix}
\ol\bu_1 \\
\ol\bv_1 
\end{pmatrix}
\cdot
(\bu_1^\sT, \bv_1^\sT)
\text{ and }
\begin{pmatrix}
-\bv_1 \\
\bu_1 
\end{pmatrix}
\otimes
(-\ol\bv_1^\sT, \ol\bu_1^\sT) =
\begin{pmatrix}
-\bv_1 \\
\bu_1 
\end{pmatrix}
\cdot
(-\ol\bv_1^\sT, \ol\bu_1^\sT),
\]
respectively.
It is then easy to verify that 
\[
\begin{pmatrix}
\ol\bu_1 \\
\ol\bv_1 
\end{pmatrix}
\cdot
(\bu_1^\sT, \bv_1^\sT) +
\begin{pmatrix}
-\bv_1 \\
\bu_1 
\end{pmatrix}
\cdot
(-\ol\bv_1^\sT, \ol\bu_1^\sT) = 
\begin{pmatrix}
\ol\bu_1\bu_1^\sT + \bv_1\ol\bv_1^\sT & \ol\bu_1\bv_1^\sT - \bv_1\ol\bu_1^\sT \\
\ol\bv_1\bu_1^\sT - \bu_1\ol\bv_1^\sT & \ol\bv_1\bv_1^\sT + \bu_1\ol\bu_1^\sT
\end{pmatrix}
\]
is a Hermitian complex matrix lying in the image of $\iota$.
In other words, this sum belongs to $\NS(X)_\bC$.
Hence, similar as in \cref{case-I}, the spectral radii of $\alpha^*|_{\NS(X)_\bC}$ and $\alpha^*|_{\End(X)_\bC}$ coincide, both equal to $|\pi_1|^2$.
Overall, we have
\[
|\omega_1|^2 = |\pi_1|^2 = \rho(\alpha^* |_{\End(X)_\bC}) = \rho(\alpha^* |_{\NS(X)_\bC}) = \rho(\alpha^* |_{\NS(X)_\bR}) = \lambda_1(\alpha).
\]
We thus conclude the proof of the equality \eqref{eq:key} in this case.
\end{case}

\begin{case}
\label{case-IV}
$D$ is of Type IV$(e_0, d)$: $e=2e_0$ and $D$ is a division algebra over the CM-field $K \supsetneq K_0$
(i.e., $K$ is a totally imaginary quadratic extension of a totally real algebraic number field $K_0$). 
Then
\[
\End(X)_\bR \isom \bigoplus_{i=1}^{e_0} \Mat_{dn}(\bC) \ \text{ and } \ \NS(X)_\bR \isom \bigoplus_{i=1}^{e_0} \sH_{dn}(\bC).
\]
For simplicity, we just deal with the case $e_0=1$.
Denote the image of $\alpha$ in $\End(X)_\bR$ by the matrix $\bA_\alpha \in \Mat_{dn}(\bC)$.
Again, the Rosati involution $\alpha^\dagger$ of $\alpha$ could be represented by the complex conjugate transpose $\bA_\alpha^* = \ol\bA_\alpha^\sT$ (see \cref{thm:NS-matrix-form}).
It follows from \cref{lemma:Kronecker}~\eqref{lemma:Kronecker2} that the induced linear map $\alpha^* |_{\Mat_{dn}(\bC)}$ on the $d^2n^2$-dimensional $\bC$-vector space $\Mat_{dn}(\bC)$ is represented by the Kronecker product $\bA_{\alpha} \otimes \overline{\bA}_{\alpha}$;
however, the induced linear map $\alpha^* |_{\End(X)_\bR}$ on the $2d^2n^2$-dimensional $\bR$-vector space $\End(X)_\bR$ is represented by the block diagonal matrix $(\bA_{\alpha} \otimes \overline{\bA}_{\alpha}) \oplus (\ol\bA_{\alpha} \otimes \bA_{\alpha})$ by \cref{lemma:Kronecker}~\eqref{lemma:Kronecker3}, though we do not need this fact later.

Note that the center of our $\bR$-algebra $\End(X)_\bR \isom \Mat_{dn}(\bC)$ is $\bC$.
Then by \cref{def:red-char,rmk:red-char}, the reduced characteristic polynomial $\chi_{\alpha}^{\reduced}(t)$ of $\alpha$ is equal to the product of the characteristic polynomial $\det(t \, \bI_{dn} - \bA_\alpha)$ of $\bA_\alpha$ and its complex conjugate.
We denote all of its complex roots by $\pi_1, \ldots, \pi_{dn}, \overline{\pi}_1, \ldots, \overline{\pi}_{dn}$.
Without loss of generality, we may assume that $\omega_1 = \pi_1$ by \cref{lemma:red-char-II}.
Let $\bv_1$ be a complex eigenvector of $\bA_{\alpha}$ corresponding to the eigenvalue $\pi_1$.
Then $\bv_1 \otimes \ol\bv_1$ is an eigenvector of $\bA_{\alpha} \otimes \ol\bA_{\alpha}$ corresponding to the eigenvalue $\pi_1 \ol\pi_1 = |\pi_1|^2$.
Note that $\bv_1 \otimes \ol\bv_1$ is the associated column vector of the Hermitian complex matrix $\ol\bv_1 \otimes \bv_1^\sT = \bv_1^\sT \otimes \ol\bv_1 \in \NS(X)_\bR$.
Hence, in this last case, we also have
\[
|\omega_1|^2 = |\pi_{1}|^2 = \rho(\alpha^* |_{\Mat_{dn}(\bC)}) = \rho(\alpha^* |_{\NS(X)_\bR}) = \lambda_1(\alpha).
\]
\end{case}

We thus finally complete the proof of \cref{thmA}.
\end{proof}

\begin{remark}
\label{final-remark}
\begin{enumerate}[(1)]
\item It follows from our proof, in particular from the key equality \eqref{eq:key},
as well as Birkhoff's generalization of the Perron--Frobenius theorem,
that either $\omega_2 = \omega_1 \in \bR$ or $\omega_2 = \ol{\omega}_1 \neq \omega_1$.
This is true for any complex torus $X$ because by the Hodge decomposition we have $H^1(X, \bC) = H^{1,0}(X) \oplus \ol{H^{1,0}(X)}$, where $H^{1,0}(X) = H^0(X, \Omega_X^1)$.
A natural question is whether it is true for all $\omega_i$ in general, i.e., either $\omega_{2i} = \omega_{2i-1} \in \bR$ or $\omega_{2i} = \ol{\omega}_{2i-1} \neq \omega_{2i-1}$ for any $2 \le i \le g = \dim X$.

\item If our self-morphism $f$ is not surjective or $\alpha$ is not an isogeny, one can also proceed by replacing $X$ by the image $\alpha(X)$, which is still an abelian variety of dimension less than $\dim X$.
\end{enumerate}
\end{remark}


\phantomsection
\addcontentsline{toc}{section}{Acknowledgments}
\noindent \textbf{Acknowledgments. }
I would like to thank Dragos Ghioca and Zinovy Reichstein for their constant support, Yuri Zarhin and Yishu Zeng for helpful discussions, Tuyen Trung Truong for reading an earlier draft of this article and for his inspiring comments.
Special thanks go to the referees of my another paper \cite{Hu19} since one of their comments motivates this article initially.
Finally, I am grateful to the referee for his/her many helpful and invaluable suggestions which significantly improve the exposition of the paper.




\providecommand{\bysame}{\leavevmode\hbox to3em{\hrulefill}\thinspace}
\providecommand{\MR}{\relax\ifhmode\unskip\space\fi MR }
\providecommand{\MRhref}[2]{%
  \href{http://www.ams.org/mathscinet-getitem?mr=#1}{#2}
}
\providecommand{\href}[2]{#2}

\end{document}